\documentclass{article}
\usepackage{amsmath}
\usepackage{amsfonts}
\usepackage{amssymb}
\usepackage{graphicx}%
\usepackage{tkz-graph}
\usepackage{amsmath,amssymb,tikz-cd}
\newtheorem{theorem}{Theorem}

\newtheorem{definition}[theorem]{Definition}
\newtheorem{example}[theorem]{Example}

\newtheorem{lemma}[theorem]{Lemma}

\newtheorem{proposition}[theorem]{Proposition}

\newenvironment{proof}[1][Proof]{\noindent\textbf{#1.} }{\ \rule{0.5em}{0.5em}}
\newcommand{\bpartial}{\mathop{\partial\kern -4pt\raisebox{.8pt}{$|$}}}
\newcommand{\bra}{\mathopen{[\kern-1.6pt[}}
\newcommand{\ket}{\mathclose{]\kern-1.5pt]}}
\newcommand{\bbra}{\mathopen{[\kern-2.2pt[\kern-2.3pt[}}
\newcommand{\bket}{\mathclose{]\kern-2.1pt]\kern-2.3pt]}}

\makeindex
\begin{document}
\title {\large{\bf Poisson-Nijenhuis Structure on Lie groupoids from the Invariance's Point of View}}
\vspace{2mm}
\author { \small{ \bf Gh. Haghighatdoost$^1$ }\hspace{-2mm}{\footnote{Corresponding author,  e-mail: gorbanali@azaruniv.ac.ir}} , \small{ \bf J. Ojbag$^2$}\hspace{-1mm}{ \footnote{ e-mail:j.ojbag@azaruniv.ac.ir}}  \\
{\small{$^{1,2}$\em Department
of Mathematics,Azarbaijan Shahid Madani University, 53714-161, Tabriz, Iran}}\\ }
 \maketitle
 
\begin{abstract}
In this paper, we introduce right-invariant (similarly, left-invariant) Poisson-Nijenhuis  Structures on Lie groupoids and their infinitesimal counterparts as called $(\Lambda , \mathbf{n})-$structures. Also, we present a mutual correspondence between $(\Lambda ,\mathbf{n})-$structures on Lie algebroids with Poisson-Nijenhuis structures $(\Pi , \mathbf{N})$ on their Lie groupoids under some conditions. Also, we will construct Poisson-Nijenhuis pair groupoids as an example from the invariance's point of view.
\end{abstract}

\section{Introduction} 
Poisson-Lie groups introduced by Drinfel’d \cite{Drinfeld}. Recently many researchers working on geometric structures on Lie groupoids and try to extend known methods on Lie groups to Lie groupoids. By linearization a Lie groupoid at the units, one can correspond a Lie algebroid to it. Suppose that $G \rightrightarrows M$ be a Lie groupoid with source and target maps $s$ and $t$. We denote it's Lie algebroid by $AG$, equipped with anchor map $\rho$ and bracket $[.,.]$.\\
In section 2 we review some preliminary definitions and theorems about Poisson-Nijenhuis manifolds and groupoids. Moreover we remember some basic subjects about Lie groupoids and Lie algebroids (for more details refer to \cite{F0Magri},\cite{Kosmann},\cite{Kosman2},\cite{Kiril0Mackenzie}).\\
In section 3 we define right-invariant Poisson-Nijenhuis structure on Lie groupoid and infinitesimal counterpart of this, called algebraic structures corresponding to Poisson-Nijenhuis Lie groupoids.
Also, we prove that $(\Pi , \mathbf{N})$-structures on Lie groupoids are in one-to-one correspondence with $(\Lambda , n)$-structures on their Lie algebroids from the Invariance's Point of View.\\
In section 4 we construct right-invariant Poisson-Nijenhuis structure on the pair groupoids. The one-to-one correspondence between $(\Lambda , n)$-structures and right-invariant Poisson-Nijenhuis structure on the pair groupoids has been verified.
\section{Preliminaries}
\begin{definition}
A Lie groupoid consists of two smooth manifolds $G$ and $M$, together with a set of smooth morphisms $\{s, t, u, i, m\}$, namely:
\begin{itemize}
\item The source map $s: G \to M$, which is a surjective submersion;
\item The target map $t: G \to M$, which is a surjective submersion;
\item The unit map $u: M \to G$, which is an embedding and is denoted by $1_x := u(x)$ for $x \in M$. Each $u(x)$ is called a unit;
\item The inversion $i: G \to G$, which is a diffeomorphism and is denoted by $g^{-1}:= i(g)$ for $g \in G$;
\item The multiplication $m: G^2 \to G$ denoted by $(g, h) \to  g.h := m(g, h)$ is smooth, where $G^2=G _s\times_t G=\{(g,h) \in G\times G|s(g)=t(h)\}$.
\end{itemize}
These are subject to the conditions that:
\begin{itemize}

\item[1.] $s(gh)=s(h), t(gh)=t(g)$.
\item[2.] $(gh)k=g(hk)$, whenever $s(g)=t(h)$ and $s(h)=t(k)$.
\item[3.] $1_{t(g)}g=g=g1_{s(g)}$
\item[4.] $gg^{-1}=1_{t(g)}, \;g^{-1}g=1_{s(g)}$.

\end{itemize}
\end{definition}
The main difference between Lie groups and Lie groupoids is that for groups, we could compose any two elements, while for Lie groupoids
this is only possible if the first arrow ends where the second one starts.\\ 
If the base $M$ is a point, a Lie groupouid $G \rightrightarrows M$ is the Lie group.\\
Consider Lie groupoid $G \rightrightarrows M$, for all $x \in M$, $s^{-1}(x)$ is called its source-fibre or $s$-fibre, $G_x := s^{-1}(x) \cap t^{-1}(x)$ its isotropy group and $L_x := t(s^{-1}(x))$ its orbit. $L_x \subset M$ is an embedded submanifold of $G$.
$G$ is called transitive if it has only one orbit. Its orbit space is a single point. The pair groupoid $M \times M \rightrightarrows M$ is an important example of a transitive Lie groupoid.
\begin{definition}
For a Lie groupoid $G \rightrightarrows M$,
\begin{itemize}

 \item $G \rightrightarrows M$ source-connected ($s$-connected), if $s^{-1}(x)$ is connected for each $x \in M$,
 \item $G \rightrightarrows M$ source-simply-connected ($s$-simply connected), if $s^{-1}(x)$ is connected and simply-connected for each $x \in M$.

\end{itemize}
\end{definition}
\begin{definition}\label{def2}
Let Lie groupoids $G \rightrightarrows M ,\; G^\prime \rightrightarrows M^\prime$. A morphism between Lie groupoids is a pair of maps $F: G \to G^\prime \;,\;f:M \to M^\prime $ such that:
\begin{itemize}
\item
$s^\prime \circ F=f \circ s,\;\;t^\prime \circ F=f \circ t$,
\item
$F(hg)=F(h)F(g),\;\;\forall (h,g)\in G \ast G$.
\end{itemize}
If $F$ and (hence) $f$ are diffeomorphisms, the morphism of groupoids called isomorphism of Lie groupoids.
\end{definition}

\begin{definition}
A Lie algebroid is a vector bundle $A$ on base $M$ together with a bracket of sections $\Gamma A \times \Gamma A \to \Gamma A$ and a map $\rho: A \to TM$ such that

\begin{itemize}
\item the bracket of sections makes $\Gamma A$ an $R-$Lie algebra,
\item $[X,fY]=f[X,Y]+\rho (X)(f)Y, \forall X,Y \in \Gamma A, \; f\in C^\infty(M)$,
\item $\rho[X,Y]=[\rho X,\rho Y], \;\; X,Y\in \Gamma A$.
\end{itemize} 
\end{definition}
For a Lie groupoid $G\rightrightarrows M$, restrict $TG$ to the identity elements; get $T_{1M}G$, a vector bundle on $M$. Right-translations $R_g$, map $s-$ fibers to $s-$ fibers. So take the kernel of $T(s): T_{1M}G \to TM$. Call this $AG$.
Each $X\in \Gamma AG$ defines a right-invariant vector field $\overrightarrow{X}$ on $G$ by $\overrightarrow{X}(g)=Xg$. That is, $\overrightarrow{X}$ is $s-$vertical and $\overrightarrow{X}(hg)=\overrightarrow{X}(h)g$ for all $h,g$.
Each right-invariant vector field is $\overrightarrow{X}$ for some $X\in \Gamma AG$. The bracket of right-invariant vector fields is right-invariant. Define bracket on $\Gamma AG$ by $\overrightarrow{[X , Y]}=[\overrightarrow{X},\overrightarrow{Y}]$.  $AG$ is the Lie algebroid of $G$.

\begin{definition}
Let $M$ be a smooth manifold and $N:TM \to TM$ be a vector valued $1-$form, or a $(1,1)-$tensor on $M$. Then its Nijenhuis torsion $\tau N$ is a vector valued $2-form$ deﬁned by 
\[\tau N(X,Y):=[NX,NY ]-N([NX,Y]+[X,NY ]-N[X,Y ]), \; for X,Y\in \Gamma(TM).\]
An $(1,1)$-tensor $N$ is called a Nijenhuis tensor if its Nijenhuis torsion vanishes.
\end{definition}
Given a Nijenhuis tensor $N$, one can define a new Lie algebroid structure on $TM$ deformed by $N$. We denote this Lie algebroid by $(TM)_N$. The bracket $[ , ]_N$ and anchor $id_N$ of this deformed Lie algebroid are given by
\[[X,Y]_N=[NX,Y]+[X,NY ]-N[X,Y ]\; \text{and}\; id_N=N, \;\forall X,Y\in \Gamma(TM).\]
Let $M$ be a smooth manifold and $\Pi \in \Gamma (\wedge^2 TM)$ a bivector field on $M$. Then one can define a skew-symmetric bracket $[, ]_\Pi$ on the space $\Omega^1(M)$ of $1$-forms on $M$, given by
\[[\alpha,\beta]_\Pi :=\mathcal{L}_{\Pi^{\sharp}\alpha}\beta-\mathcal{L}_{\Pi^{\sharp}\beta}\alpha-d(\Pi(\alpha,\beta)), \forall \alpha,\beta \in \Gamma(T^{\ast}M),\] where $\Pi^{\sharp}:T^\ast M \to TM$, $\alpha \to \Pi(\alpha , -)$ is the bundle map induced by $\Pi$. If $\Pi$ is a Poisson bivector (that is, $[\Pi,\Pi]=0$), then the cotangent bundle $T^\ast M$ with the above bracket and the bundle map $\Pi^{\sharp}$ forms a Lie algebroid. We call this Lie algebroid as the cotangent Lie algebroid of the Poisson manifold $(M,\Pi)$ and is denoted by $(T^\ast M)_\Pi$. $(T^\ast M)_\Pi$ knows everything about $(M,\Pi)$ . 

\begin{definition}\label{2.6}
A Poisson-Nijenhuis manifold is a manifold $M$ together with  a Poisson bivector $\Pi \in \Gamma(\wedge^2 TM)$ and a Nijenhuis tensor $N$ such that  they are compatible in the following senses:
\begin{itemize}
\item $N \circ \Pi^{\sharp}=\Pi^{\sharp} \circ N^\ast$ (thus, $N \circ \Pi^{\sharp}$ defines a bivector field $N \Pi$ on $M$),
\item $C(\Pi,N)\equiv 0$,
\end{itemize}
where
\[C(\Pi,N)(\alpha,\beta):=[\alpha,\beta]_{N\Pi}-([N^\ast \alpha,\beta]_\Pi+[\alpha,N^\ast \beta]_\Pi-N^\ast[\alpha,\beta]_\Pi),\; \text{for}\; \alpha, \beta \in \Omega^1(M).\]
The skew-symmetric $C^\infty (M)$-bilinear operation $C(\Pi,N)(-,-)$ on the space of $1$-forms is called the Magri-Morosi concomitant of the Poisson structure $\Pi$ and the Nijenhuis tensor $N$.\\
A Poisson-Nijenhuis manifold as above is denoted by the triple $(M,\Pi,N)$. If $\Pi$ is non-degenerate, that is, defines a symplectic structure $\omega$ on $M$, then $(M,\omega,N)$ is said to be a symplectic-Nijenhuis manifold.
\end{definition}

\begin{definition}
A Poisson groupoid is a Lie groupoid $G\rightrightarrows M$ equipped with a Poisson structure $\Pi$ on $G$, if $\Pi^{\sharp}:T^\ast G \to TG$ be a Lie groupoid morphism from the cotangent Lie groupoid $T^{\ast} G \rightrightarrows A^\ast G$ to the tangent Lie groupoid $TG \rightrightarrows TM$.


\end{definition}

\begin{definition}\label{defN}
Let $G \rightrightarrows M$ be a Lie groupoid. A multiplicative $(1,1)$-tensor $\mathbf{N}$ on the groupoid is a pair $(N,N_M)$ of $(1,1)$-tensors on $G$ and $M$, respectively, such that
\begin{center}
\begin{tikzcd}
  TG \arrow[to=2-1, shift right=0.65ex]\arrow[to=2-1, shift left=0.65ex]\arrow[to=1-2, "N"]
  & TG \arrow[to=2-2, shift right=0.65ex]\arrow[to=2-2, shift left=0.65ex] \\
   TM \arrow[to=2-2]\arrow[to=2-2, "N_M"]
  & TM

\end{tikzcd}
\end{center}
is a Lie groupoid morphism from the tangent Lie groupoid to itself. In this case, $N_M$ is the restriction of $N$ to the unit space $M$. Thus, $N_M$ is completely determined by $N$. Hence we may use $N$ to denote a multiplicative $(1,1)$-tensor.
\end{definition}

\begin{definition}
A Nijenhuis groupoid is a Lie groupoid $G\rightrightarrows M$ together with a multiplicative Nijenhuis tensor $\mathbf{N}$ (cf. Definition \ref{defN}). A Nijenhuis groupoid may also be denoted by $(G\rightrightarrows M,\mathbf{N})$.

\end{definition}

\begin{definition}
A Poisson-Nijenhuis groupoid is a Lie groupoid $G\rightrightarrows M$ together with a Poisson-Nijenhuis structure $(\Pi,N)$ on $G$ such that $(G\rightrightarrows M,\Pi)$ forms a Poisson groupoid and $(G\rightrightarrows M,N)$ a Nijenhuis groupoid. Thus, a Poisson-Nijenhuis groupoid is a Poisson groupoid $(G\rightrightarrows M,\Pi)$ together with a multiplicative Nijenhuis tensor $N:TG \to TG$ such that $(G,\Pi,N)$ is a Poisson-Nijenhuis manifold. A Poisson-Nijenhuis groupoid as above is denoted by the triple $(G\rightrightarrows M,\Pi,N)$.
\end{definition}

\section{Right-invariant Poisson-Nijenhuis structures on Lie groupoids}
In this section, we define Poisson-Nijenhuis Lie groupoids from the invariant point of view, and their infinitesimal counterpart on the Lie algebroids $AG$ of $G$.
\begin{definition}\label{defp}
A Poisson-Nijenhuis structure $(\Pi,\mathbf{N})$ on a Lie groupoid $G\rightrightarrows M$ is said to be right-invariant, if:
\begin{itemize}
\item[1.] The Poisson structure $\Pi$ is right invariant, i.e., there exists $\Lambda \in \Gamma(\wedge^2 AG)$ such that $\Pi=\overrightarrow{\Lambda}$.
\item[2.] Multiplicative $(1,1)-$tensor $\mathbf{N}=(N,N_M)$ also is right-invariant, i.e., there are linear endomorphisms $n:\Gamma(AG) \to \Gamma(AG)$ and $n_M:TM \to TM$ such that
\[N=\overrightarrow{n},\;\;N_M=\overrightarrow{n_M}\]
\end{itemize}

\end{definition}
In the following we prove our claims only for $N$, beacuse $N_M$ is completely determined by $N$. This is also true for $n$ and $n_M$.
\begin{proposition}\label{pro1}
Let $(\Pi,\mathbf{N})$ be a right-invariant Poisson-Nijenhuis structure on a Lie groupoid $G \rightrightarrows M$ with Lie algebroid $AG$ and space of unites $1_M \subset G$. If $\Lambda \in \Gamma(\wedge^2 AG)$ that are the values of $\Pi$ restricted to space of unites $1_M$ and $(N|_{AG},N_M|_{TM})=\mathbf{n}$, then
\begin{itemize}
\item[1.] $[\Lambda,\Lambda]_{SN}=0$, where $[,]_{SN}$ is the Schouten-Nijehuis bracket,
\item[2.] $\mathbf{n}$ is Nijenhuis operator on $AG$, i.e., $\tau n(X,Y)=\tau n_M(X,Y)=0,\;\forall X,Y \in \Gamma AG$,
\item[3.] $\mathbf{n} \circ \Lambda^\sharp=\Lambda^\sharp \circ \mathbf{n}^\ast$
\item[4.] The Magri-Morosi concomitant's $C(\Lambda,\mathbf{n})(\alpha,\beta)=0$,
\item[5.] $\overrightarrow{\Lambda}^\sharp$ and $\overrightarrow{\mathbf{n}}$ are Lie groupoid morphisms.
\end{itemize}

\end{proposition}
\begin{proof}

1. With respect to definition \ref{defp}, since $\Pi$ is a Poisson bivector and bracket on $\Gamma (AG)$ is defined by $\overrightarrow{[X,Y]}=[\overrightarrow{X},\overrightarrow{Y}]$, hence
\[[\Pi,\Pi]=0 \longrightarrow [\overrightarrow{\Lambda},\overrightarrow{\Lambda}]=\overrightarrow{[\Lambda,\Lambda}]=0\]
2. As we know, the Nijenhuis torsion is
\[\tau n(X,Y):=\overrightarrow{[N,N ](X,Y)}=[\overrightarrow{n},\overrightarrow{n}](X,Y)=[N,N](\overrightarrow{X},\overrightarrow{Y})=0, \; \;\;\forall X,Y\in \Gamma AG.\]
The proof is same for $n_M$.\\
3. Since $\Pi^\sharp(\overrightarrow{\alpha})=\overrightarrow{\Lambda^\sharp(\alpha)}$ and $N^\ast(\overrightarrow{\alpha})=\overrightarrow{n^\ast(\alpha)}$ for $\alpha \in T^\ast A$, we have\\
\[(N \circ \Pi^\sharp-\Pi^\sharp \circ N^\ast)(\overrightarrow{\alpha})=(\overrightarrow{n} \circ \overrightarrow{\Lambda^\sharp}-\overrightarrow{\Lambda^\sharp} \circ \overrightarrow{n^\ast})(\alpha)=\overrightarrow{(n \circ \Lambda^\sharp-\Lambda^\sharp \circ n^\ast)(\alpha)}\]
4. The Magri-Morosi concomitant's defined as $C(\Pi,\mathbf{N}):\Gamma(T^\ast G) \times \Gamma(T^\ast G) \to \Gamma(T^\ast G)$ . Consider the direct product Lie groupoid $T^\ast G \times T^\ast G \rightrightarrows A^\ast \times A^\ast$.\\
The Lie algebroid of Lie subgroupoid $T^\ast G \oplus_G T^\ast G \rightrightarrows A^\ast \oplus_M A^\ast$ is isomorphic to the Lie algebroid $T^\ast A \oplus_A T^\ast A \rightrightarrows A^\ast \oplus_M A^\ast$ \cite{Apurba0Das}. We can consider the Magri-Morosi concomitant as a Lie groupoid morphism from $T^\ast G \oplus_G T^\ast G \rightrightarrows A^\ast \oplus_M A^\ast$ to $T^\ast G \rightrightarrows A^\ast$, therefore we have
\[C(\Lambda ,\mathbf{n})=j^\prime_G \circ \mathcal{L}(C(\Pi,\mathbf{N})) \circ (j^\prime_G \oplus j^\prime_G)^{-1}:T^\ast A \oplus_A T^\ast A \to T^\ast A\]
Where $j: T(AG) \to A(TG)$ and $j^\prime: A(T^\ast G) \to T^\ast(AG)$ are canonical isomorphisms \cite{Kiril0Mackenzie2},\cite{Kiril0Mackenzie3}.\\
Since $C(\Pi ,\mathbf{N})=0$, therefore $C(\Lambda ,\mathbf{n})=0$.\\
According to the definitions \ref{defN} and \ref{defp} (5) is obviuos.
\end{proof}
\\
Similar to the case of lie algebras we can find a linear isomorphism between lie algebroid and tangent space of corresponding Lie groupoid.
\begin{lemma}\label{lem0}
Left invariant vector fields on Lie groupoid $G \rightrightarrows M$ identified with section of Lie algebroid $AG$.
\end{lemma}
\begin{proof}\cite{Kiril0Mackenzie}
\end{proof}
\\
Suppose $G \rightrightarrows M$ is a Lie groupoid with Lie algebroid $AG$ and $\Lambda \in \Gamma (\wedge^2 AG)$ is a bivector, $n:\Gamma(AG) \to \Gamma(AG)$ is a linear endomorphism on $AG$ that satisfies in (1-5) conditions of proposition \ref{pro1}. In the following theorem we proof that $(\overrightarrow{\Lambda},\overrightarrow{n})$ define a Poisson-Nijenhuis structure on Lie groupoid $G$. The $s-$connected and $s-$simply connected for our intended examples of Lie groupoids are necessary, beacuse of we need these to generate Lie groupoid $G \rightrightarrows M$ by flows of exponential map.
\begin{theorem}\label{theorem1}
Let $s-$connected and $s-$simply connected Lie groupoid $G \rightrightarrows M$ with Lie algebroid $AG$. Consider $ \Lambda \in \Gamma(\wedge^2AG)$ be an element  satisfying $ [\Lambda ,\Lambda]=0$. Then $\Pi =\overrightarrow{\Lambda}$ defines a Poisson groupoid structure on $G$.\\
Furthermor for the endomorphisms $n:\Gamma(AG) \to \Gamma(AG)$ and $n_M: TM \to TM$ there exists multiplicative $(1,1)$-tensors $N:TG \to TG,\;N_M:TM \to TM $ such that $ \overrightarrow{n}=N,\;\overrightarrow{n_M}=N_M $ and compatible with $\Pi$. 
\end{theorem}
\begin{proof}
According to the definition of the bivector $\Pi$ we have
\[[\Pi,\Pi] = [\overrightarrow{\Lambda},\overrightarrow{\Lambda}]=\overrightarrow{[\Lambda ,\Lambda]}\]
Since $[\Lambda ,\Lambda]=0$, therefore $[\Pi,\Pi] =0$.\\
It only remains to show $\Pi^{\sharp}$ is a Lie groupoid morphism from the cotangent Lie groupoid $T^{\ast} G \rightrightarrows (AG)^{\ast}$ to the tangent Lie groupoid $TG \rightrightarrows TM$. From definition 
\[\Pi^\sharp(\overrightarrow{\alpha})=\overrightarrow{\Lambda}^\sharp(\alpha)\]
and this fact that $\Lambda $ satisfy in conditions of proposition \ref{pro1}, hence according to part (5)  we have  $\overrightarrow{\Lambda}^\sharp$ is Lie groupoid morphism. The following diagram, show that $\Pi^\sharp$ is morphism of groupoids:
\begin{center}
\begin{tikzcd}
  T^\ast G \arrow[to=1-2, shift right=0.65ex]\arrow[to=1-2, shift left=0.65ex]\arrow[to=2-1, "\Pi^\sharp"]
  & A^\ast G \arrow[to=2-2, "\rho_\ast"] \\
   TG \arrow[to=2-2, shift right=0.65ex]\arrow[to=2-2, shift left=0.65ex]
  & TM 

\end{tikzcd}
\end{center}
 Hence $(G, \Pi )$ is Poisson groupoid and this will complete the first part of the proof.\\
According to lemma \ref{lem0} every section of Lie algebroid $AG$ can be extended to a right-invariant vector field on Lie groupoid $G$. Hence the endomorphisms on the sections of Lie algebroid $AG$ can be extended to endomorphisms on tangent space $TG$ such that send right-invariant vector field to itself. Therefore we have $\overrightarrow{n}=N$ (Similarly, we have $ \overrightarrow{n_M}=N_M$).
We can define an isomorphism between sections of Lie algebroid and right invariant vector fields on $G$. We show this isomorphism by $\lambda_\alpha$. Therefore we have the commutative diagram as follow:

\begin{center}
\begin{tikzcd}
  \Gamma (AG) \arrow[to=2-1]\arrow[to=1-2, "n"]\arrow[to=2-1, "\lambda_\alpha"]
  & \Gamma (AG) \arrow[to=2-2]\arrow[to=2-2, "\lambda_\alpha"] \\
   TG \arrow[to=2-2, "N"]
  & TG 

\end{tikzcd}
\end{center}
For expressing the compatibility of a pair $(\Pi , N)$, where $\Pi=\overrightarrow{\Lambda}$ is a Poisson bivector and $N=\overrightarrow{n}$ is multiplicative $(1,1)-$tensor, we verify  $N\Pi=\Pi N$ and $N\Pi^\sharp=\Pi^\sharp N^t$ or equivalently $N \Pi - \Pi N=0$ and $N \circ \Pi^\sharp-\Pi^\sharp \circ N^\ast=0$ as follow:
\[N \Pi - \Pi N=\overrightarrow{n} \overrightarrow{\Lambda} - \overrightarrow{\Lambda} \overrightarrow{n}=\overrightarrow{n\Lambda} -\overrightarrow{\Lambda n} \]
Since $n$ and $\Lambda$ are compatible, therefore $n\Lambda = \Lambda n$. Hence
\[\overrightarrow{n\Lambda} -\overrightarrow{\Lambda n}=\overrightarrow{\Lambda n}-\overrightarrow{\Lambda n}=0\]
Also we have
\[N \circ \Pi^\sharp-\Pi^\sharp \circ N^\ast=\overrightarrow{n} \circ \overrightarrow{\Lambda^\sharp}-\overrightarrow{\Lambda}\circ \overrightarrow{n}^\ast=\overrightarrow{n \circ \Lambda^\sharp}-\overrightarrow{\Lambda \circ n^\ast}\]
With respect to part 3 of proposition \ref{pro1} we have $n \circ \Lambda^\sharp=\Lambda \circ n^\ast$, therfore 
\[\overrightarrow{n \circ \Lambda^\sharp}-\overrightarrow{\Lambda \circ n^\ast}=\overrightarrow{n \circ \Lambda^\sharp}-\overrightarrow{n \circ \Lambda^\sharp}=0\]
The concomitant Magri-Morosi $C(\Pi,N)(\overrightarrow{\alpha},\overrightarrow{\beta})=-\overrightarrow{C(\Lambda ,n)(\alpha , \beta)}$, hence $C(\Lambda ,n)=0$ gives $C(\Pi,N)=0$.
\end{proof}
$(\Lambda , n)$ which satisfy condition (1-5) of proposition \ref{pro1}; so-called \textit{$(\Lambda , n)-$ structure} on the Lie algebroid $AG$ and $(\overrightarrow{\Lambda},\overrightarrow{n})$ is right-invariant Poisson-Nijenhuis structure on $G$. similarly, we can define left-invariant Poisson-Nijenhuis structures. It is sufficient to replace right by left.

\section{Example}
In this section we consider the groupoid $G \rightrightarrows M$ with Lie algebroid $A$. As we know the tangent Lie groupoid of  the groupoid $G \rightrightarrows M$ is $TG \rightrightarrows TM$ whose structural maps are all obtained by taking the derivatives of the structural maps of $G$ and the cotangent bundle of a Lie groupoid $G \rightrightarrows M$ also carries a natural Lie groupoid structure, $T^\ast G \rightrightarrows A^\ast  $, where $A^\ast $ is the dual vector bundle to $A$. The source and target maps $\tilde{s}, \tilde{t}: T^\ast G \to A^\ast  $ are defined by the restriction of covectors to the subspaces tangent to the $s-$fibers and $t-$fibers, respectively \cite{Bursztyn}.\\
In this section, we define a right-invariant Poisson-Nijenhuis structure for Piar groupoids. Also, we show that there is a one-to-one correspondence between the Poisson-Nijenhuis Lie algebroid and corresponding Lie groupoid from the invariance's point of view. 

\begin{example}
Any manifold $M$ gives rise to another Lie groupoid called the pair groupoid of $M$ which is indicated by  $M \times M \rightrightarrows M$ and a pair $(x, y)$ is seen as an arrow from $x$ to $y$ and the composition is given by:
\[ (x, y).(y, z)  = (x, z)\]
Clearly, one has $1_x = (x, x)$ and $(x, y)^{-1}=(y, x)$.\\
The Lie algebroid of the pair groupoid $M \times M \rightrightarrows M$ is $TM \to M$ with the usual Lie bracket of vector fields and the identity map as the anchor.\\
Let $\Pi$ be a Poisson bivector on Lie groupid $M\times  M \rightrightarrows M$. Since there are two functions 
\[\Pi^\sharp:T^\ast(M\times M) \to T(M\times M),\;id:A^\ast  \to TM\] such that the diagram

\begin{center}
\begin{tikzcd}
  T^\ast(M\times M) \arrow[to=2-1, shift right=0.65ex]\arrow[to=2-1, shift left=0.65ex]\arrow[to=1-2, "\Pi^\sharp"]
  & T (M \times M) \arrow[to=2-2, shift right=0.65ex]\arrow[to=2-2, shift left=0.65ex] \\
   A^\ast \arrow[to=2-2]\arrow[to=2-2, "id"]
  & TM

\end{tikzcd}
\end{center}
is commutative. The existence of these functions concludes from this fact that says: for every Lie groupoid $ G \rightrightarrows M $, there is a natural morphism of Lie groupoids $ G \rightrightarrows M \times M, \; g \mapsto (s (g), t (g)) $. Also, every Poisson manifold $(M,\Pi)$ can be seen as a Lie algebroid $A$ with a vector bundle isomorphism $A \simeq T^\ast M$ satisfying certain properties \cite{Luca}.\\
Hence $\Pi^\sharp: T^\ast (M \times M) \to T(M \times M)$ is a Lie groupoid morphism between $T^\ast(M\times M) \rightrightarrows A^\ast $ and $T(M\times M) \rightrightarrows TM$, where $A^\ast $ is dual of Lie algebroid of $M\times M \rightrightarrows M$. Therefore the Lie groupoid $M \times M \rightrightarrows M$ is a Poisson groupoid.\\
The existence of Nijenhuis tensor on pair groupoid $M \times M \rightrightarrows M$ is obvious. Consider $N=id$ and $N_M$ consider the function that project the tangent space to the base space as follow:
\begin{center}
\begin{tikzcd}
  T^\ast(M\times M) \arrow[to=2-1, shift right=0.65ex]\arrow[to=2-1, shift left=0.65ex]\arrow[to=1-2, "id"]
  & T (M \times M) \arrow[to=2-2, shift right=0.65ex]\arrow[to=2-2, shift left=0.65ex] \\
   TM \arrow[to=2-2]\arrow[to=2-2, "Proj"]
  & M

\end{tikzcd}
\end{center}
Examining the terms of the definition \ref{def2} is simple.\\
Examining the conditiones of the definition \ref{2.6} are as follow:\\
Let $f \in T^\ast (M \times M)$, then we have:
\[(\mathbf{N} \circ \Pi^\sharp)(f)=\mathbf{N}(\Pi^\sharp(f))=(id(\Pi^\sharp(f)),Proj (\Pi^\sharp(f)))\]
On the other hand
\[(\Pi^\sharp \circ N^\ast)(f)=(\Pi^\sharp\circ \mathbf{N})(f)=\Pi^\sharp(\mathbf{N}(f))=\Pi^\sharp(id(f),Proj(f))\]
With respect to this fact that $\Pi^\sharp(f)$ and $\mathbf{N}$ are Lie group homomorphisms, we have: 
\[\mathbf{N}\circ \Pi^\sharp = \Pi^\sharp \circ \mathbf{N}^\ast .\]
It is sufficient to show that $C(\Pi , \mathbf{N})=0$, this is true because of we consider $\mathbf{N}=(id,Proj)$ and usual bracket.\\
Hence $(M\times M \rightrightarrows M, \Pi, \mathbf{N})$ is Poisson-Nijenhuis Lie groupoid.\\
As we mentioned above the Lie algebroid of $M \times M \rightrightarrows M$ is $TM \to M$ with identity map as anchor and usual bracket, we use $A$ for it.\\
Now according to lemma \ref{lem0} we can consider $\Lambda \in \Gamma (A)$ and an arbitarary endomorphism such as $n$ on algebroid $A$ and extend these to the right (left) invariant vector field $\overrightarrow{\Lambda},\overrightarrow{n} $ on groupoid. With respect to theorem \ref{theorem1} $(\overrightarrow{\Lambda},\overrightarrow{n}) $ are Poisson-Nijenhuis structure on pair groupoid.
\end{example}
\bigskip
\providecommand{\bysame}{\leavevmode\hbox
to3em{\hrulefill}\thinspace}

\end{document}